\documentclass[12pt]{amsart} 
\usepackage{latexsym,fancyhdr,amssymb,color,amsmath,amsthm,graphicx,listings,comment}

\newtheorem{thm}{Theorem}       
       
\usepackage[section]{placeins}      \let\paragraph\subsection
\title{More on Poincare-Hopf and Gauss-Bonnet}
\author{Oliver Knill} \date{12/1/2019, Last update 12/21/2019}
\address{Department of Mathematics \\ Harvard University \\ Cambridge, MA, 02138 }

\begin{document}

\begin{abstract}
We illustrate connections between differential geometry on finite simple graphs 
$G=(V,E)$ and Riemannian manifolds $(M,g)$. The link is that curvature can be defined 
integral geometrically as an expectation in a probability space 
of Poincar\'e-Hopf indices of coloring or Morse functions. 
Regge calculus with an isometric Nash embedding links then the Gauss-Bonnet-Chern integrand
of a Riemannian manifold with the graph curvature. There is also a 
direct nonstandard approach \cite{Nelson77}: if $V$ is a finite set 
containing all standard points 
of $M$ and $E$ contains pairs which are closer than some positive number.
One gets so finite simple graphs $(V,E)$ which leads to the standard curvature.
The probabilistic approach is an umbrella framework which covers discrete 
spaces, piecewise linear spaces, manifolds or varieties.
\end{abstract} 
\maketitle

\section{Poincar\'e-Hopf}

\paragraph{}
For a finite simple digraph $(V,E)$ with no triangular cycles,
we can define the {\bf index} $i(v) = 1-\chi(S^-(v))$, where $S^-(v)$ is
the graph generated by all vertices pointing towards $v$ and 
where the {\bf Euler characteristic} $\chi(G)=\sum_{x \subset G} \omega(x)$ sums
$\omega=(-1)^{{\rm dim}(x)}$ over the set $G$ of all complete subgraphs $x$. 
As usual, we identify here $G$ with the Whitney complex defined by $(V,E)$. 
The following result appeared already in \cite{PoincareHopfVectorFields} and
is a discrete analog of \cite{poincare85,hopf26,Spivak1999}. 

\begin{thm}[Poincar\'e-Hopf for digraphs]
$\sum_{ v \in V} i(v) = \chi(G)$.
\end{thm}

\begin{figure}[!htpb]
\scalebox{0.3}{\includegraphics{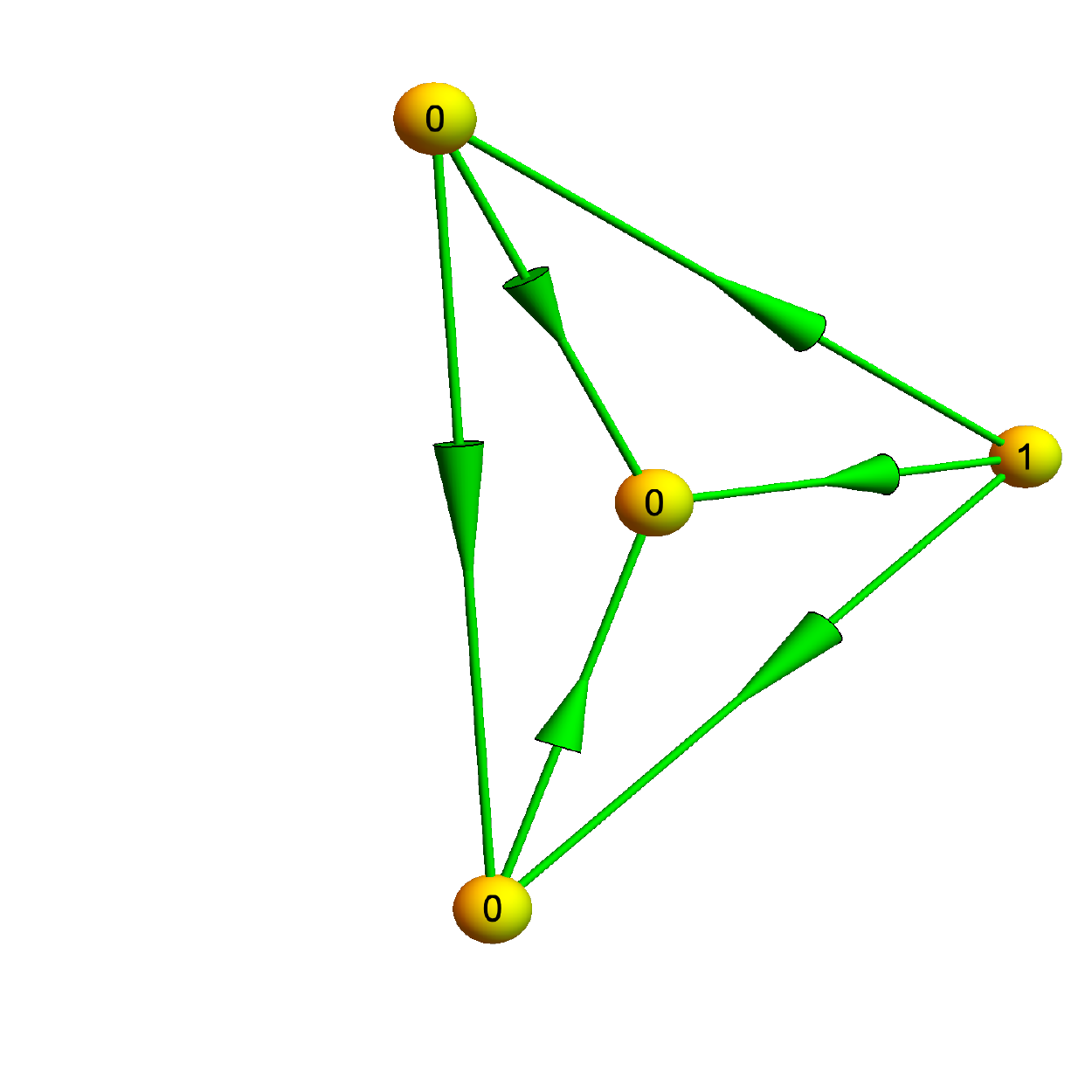}}
\scalebox{0.3}{\includegraphics{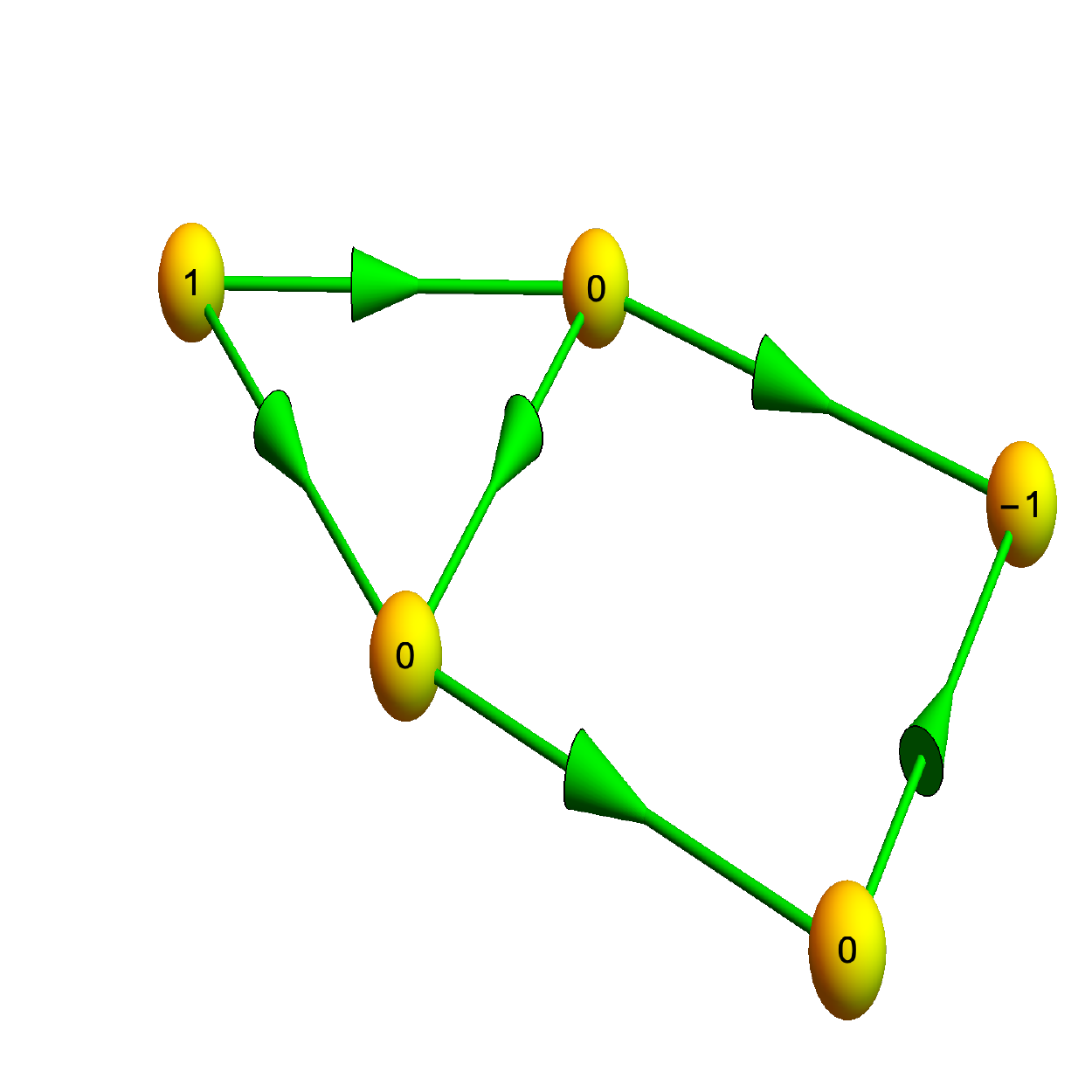}}
\scalebox{0.3}{\includegraphics{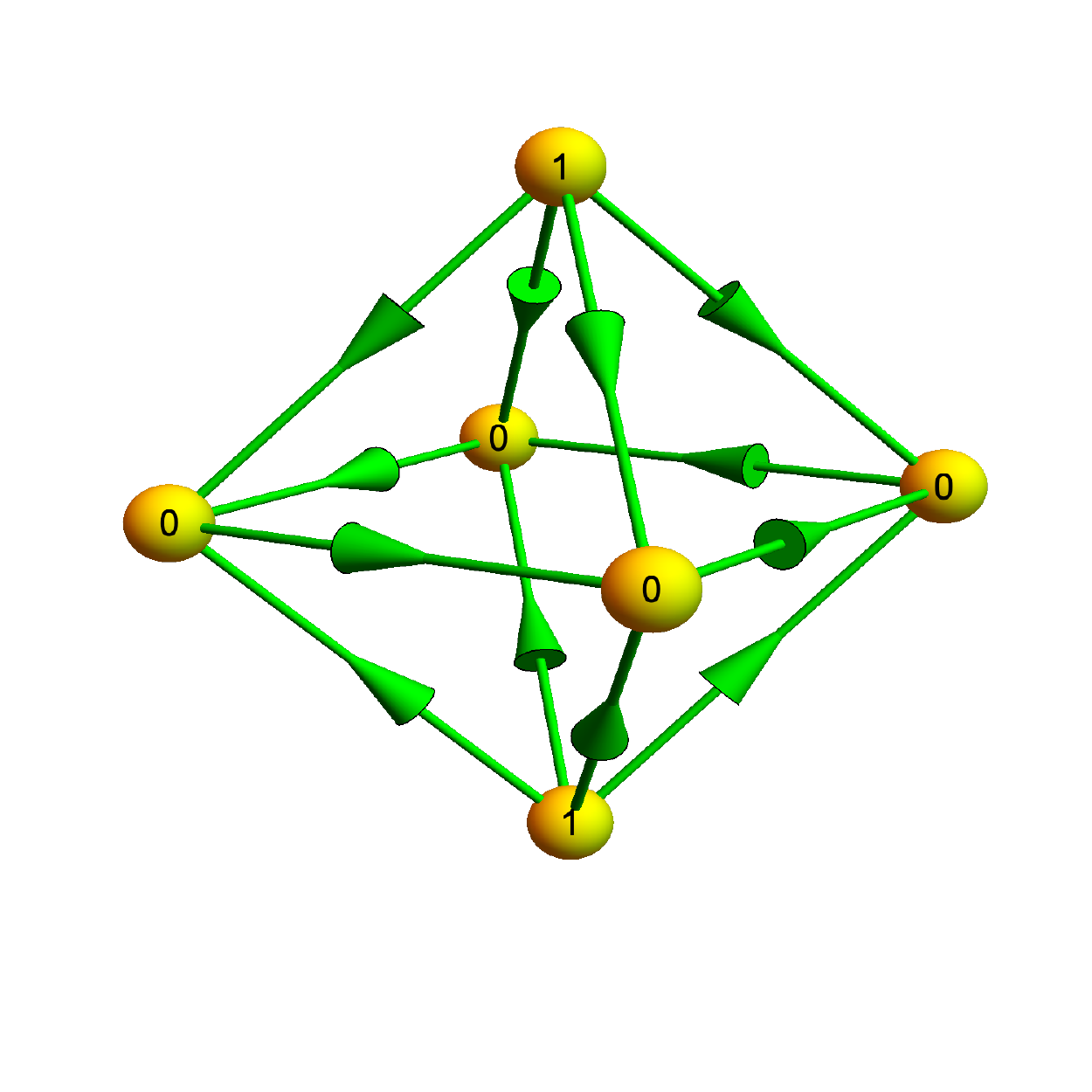}}
\label{triangle}
\caption{
Three directed graphs without circular triangles. The indices
on the vertices add up to the Euler characteristic.
}
\end{figure}

\begin{proof}
Because cyclic triangles are absent in $G$, the digraph structure defines a total 
order on each simplex $x$. Let $v=F(x)$ be the maximal element on $x$,
defining so a map $F$ from the simplicial complex $G$ to $V$. 
Because the push-forward of the signed measure $\omega(x)=(-1)^{\rm dim}(x)$ from
$G$ to $V$ is $i = F_* \omega$, we have $i(V)=\omega(G)$. 
\end{proof}

\paragraph{}
The result generalizes to functions \cite{parametrizedpoincarehopf}. 
If $f_G(t) = 1+f_0 t + \cdots + f_d t^{d+1}$
is the $f$-function of $G$, the generating function of the $f$-vector 
of $G$, then 

\begin{thm}
$f_G(t) = 1+t \sum_{v \in V} f_{S_-(v)}(t)$.
\end{thm}

\paragraph{}
The integrated version produces a Gauss-Bonnet version
\cite{dehnsommervillegaussbonnet}. 
If $F_G(t)=\int_0^t f_G(s) \; ds$ denote the anti-derivative of $f_G$,
the curvature valuation to $f$ is the anti-derivative of $f$ evaluated on the unit sphere.
The functional form of the Gauss-Bonnet formula is then
$$  f_G(t)-1 = \sum_{v \in V} F_{S(v)}(t)  \; . $$
The Euler characteristic is obtained by evaluating at $t=-1$.
This actually corresponds to the Gauss-Bonnet-Chern 
integrand in the continuum which contains Pfaffians of curvature tensor 
entries.

\begin{figure}[!htpb]
\scalebox{0.3}{\includegraphics{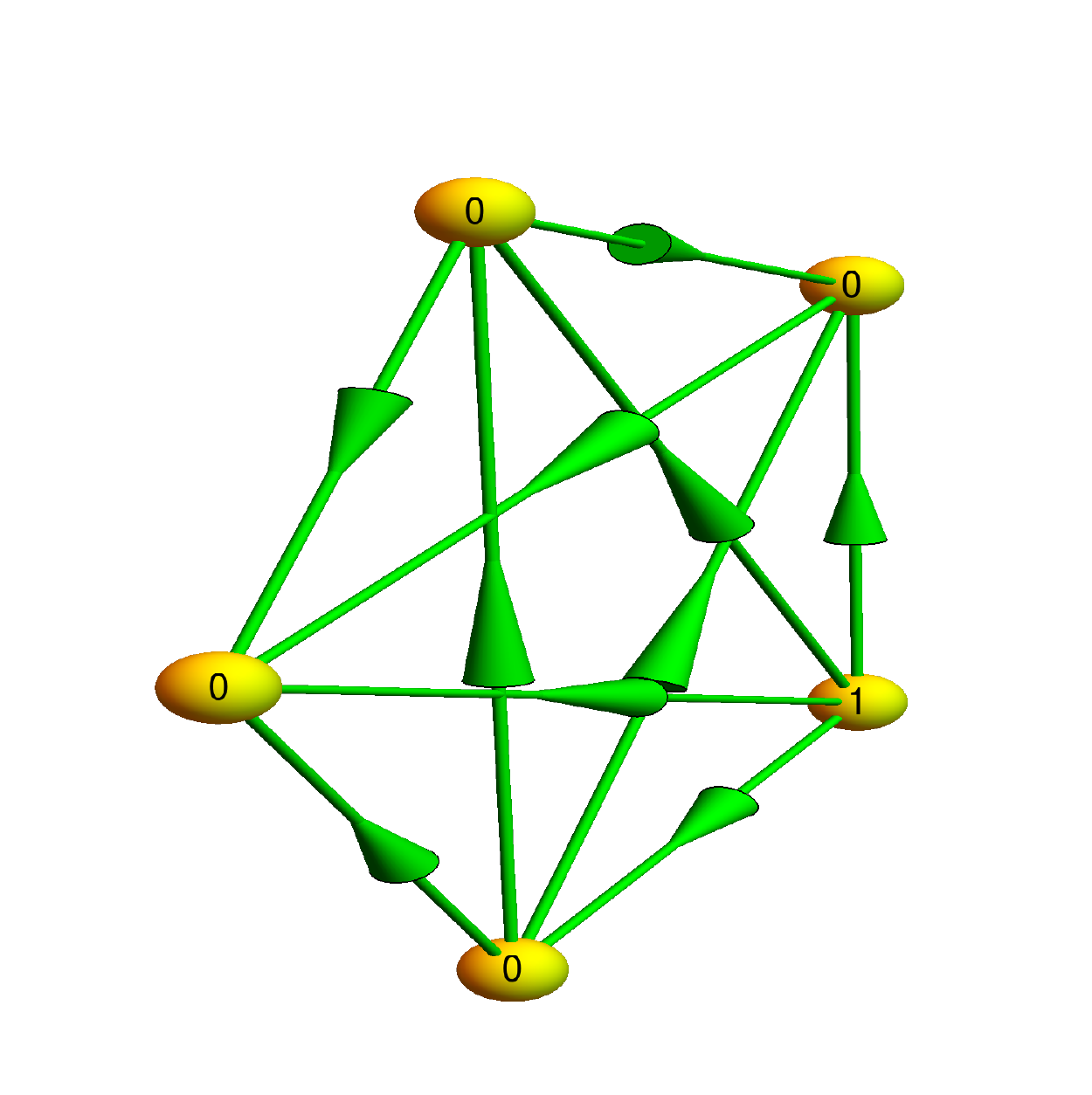}}
\scalebox{0.3}{\includegraphics{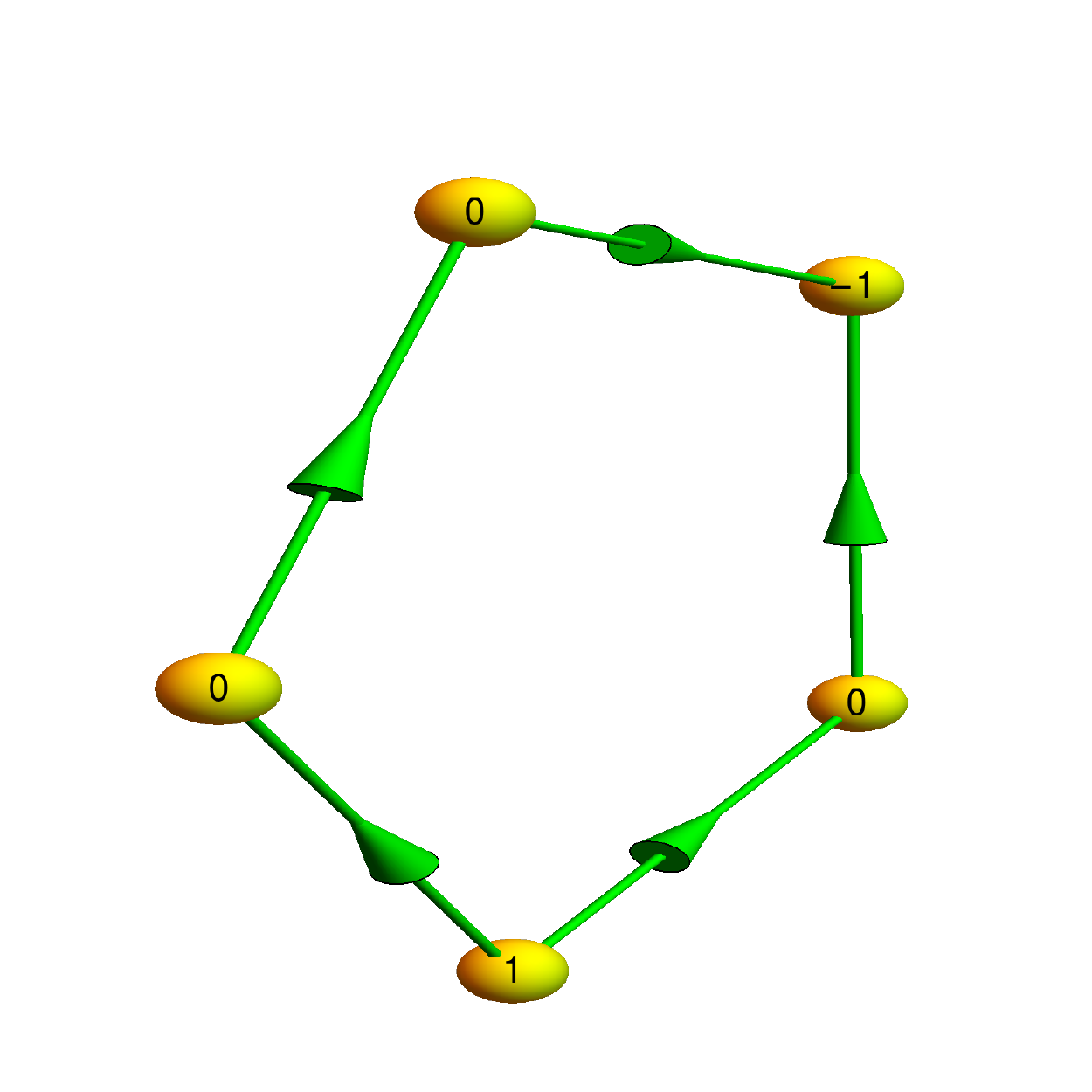}}
\scalebox{0.3}{\includegraphics{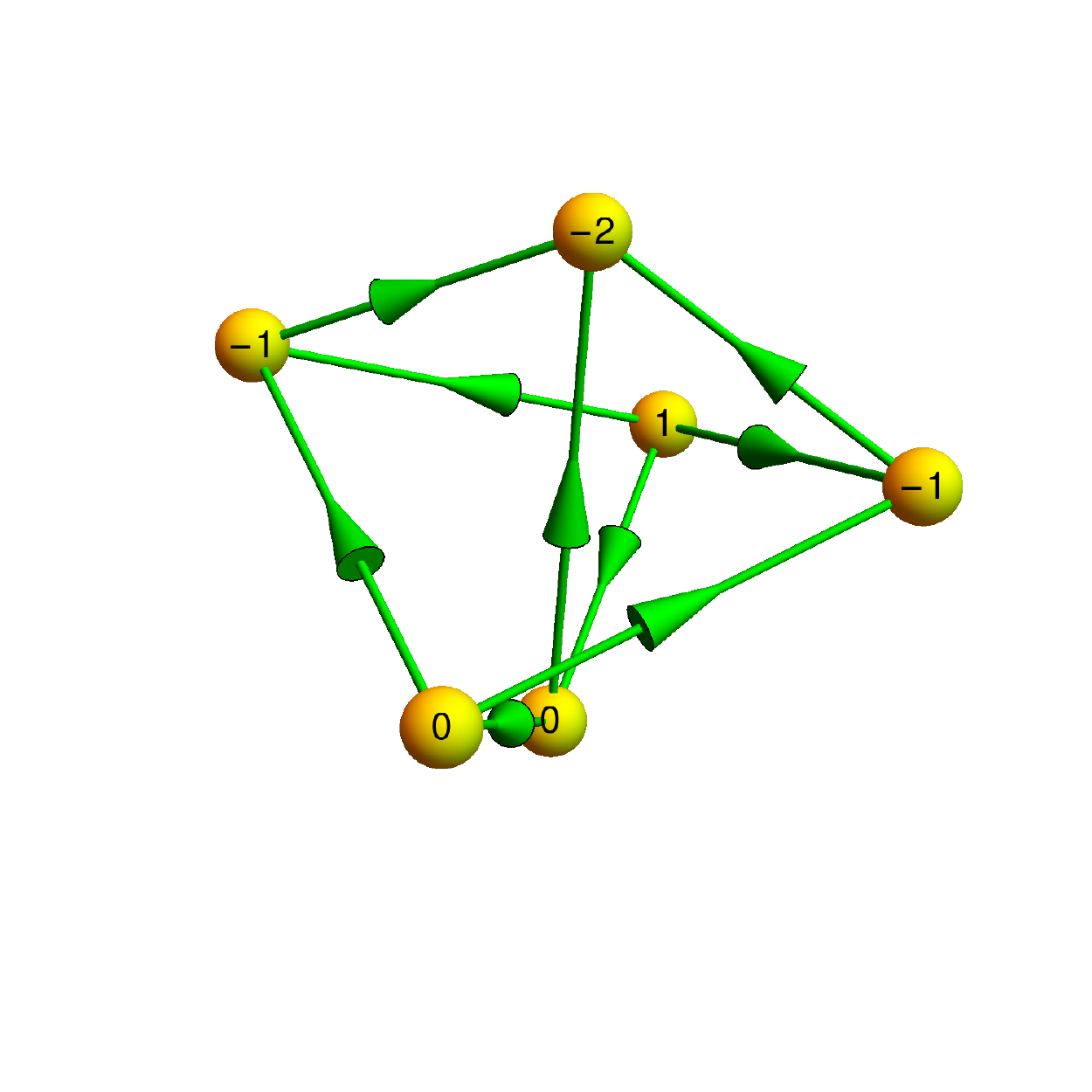}}
\label{triangle}
\caption{
Three more examples: for every complete graph, there is always just one point
with index $1$, the minimum. For cyclic graphs, the indices are either $-1$ or $1$
and there are the same number of each. For the utility graph with 
Euler characteristic $-3$, there are also examples for which all indices
are non-zero.
}
\end{figure}

\paragraph{}
The Poincar\'e-Hopf formula for the generating function $f_G(t)$ 
again relates to question how fast one can compute the $f$-vector of 
the graph. It is an NP-complete problem, as it solves the clique problem. Assuming the
clique problem is hard, we know that the following problem is hard: 
{\it how do we place an irrotational digraph structure on a graph such that
$S_-(x)$ contains about half of the vertices of the unit sphere $S(v)$ of 
every $v \in V$.}

\paragraph{}
Examples: \\
{\bf 1)} If the direction $F$ comes from a coloring $g:V \to \mathbb{R}$,
the direction is defined by $v \to w$ if $g(v)<g(w)$. The graph $S^-(v)$ is 
generated by $\{ w \in V \; | \  g(w)<g(v) \}$. See \cite{poincarehopf}. \\
{\bf 2)} For graphs without triangles or graphs
equipped with the $1$-dimensional skeleton simplicial complex, 
the index is $i(v) = 1-{\rm deg}_-(v)/2$, where ${\rm deg}_-(v)$ is the
number of incoming vertices.  \\
{\bf 3)} For $2$-graphs, graphs for which every unit sphere is a circular
graph $C_n$ with $n \geq 4$, the index $i_v(x)$ is one minus the number 
of connectivity components of $S(x)$ which come into $x$. 
If all point into $x$ (this is a sink) or all point away 
from $x$ (this is a source), then the index is $1$. With exactly two components
getting in and two parts getting out, we get a saddle. As in the continuum, the index
can not get larger than $1$. The index is $0$ if there is one incoming
direction and one outgoing direction. 

\begin{figure}[!htpb]
\scalebox{0.45}{\includegraphics{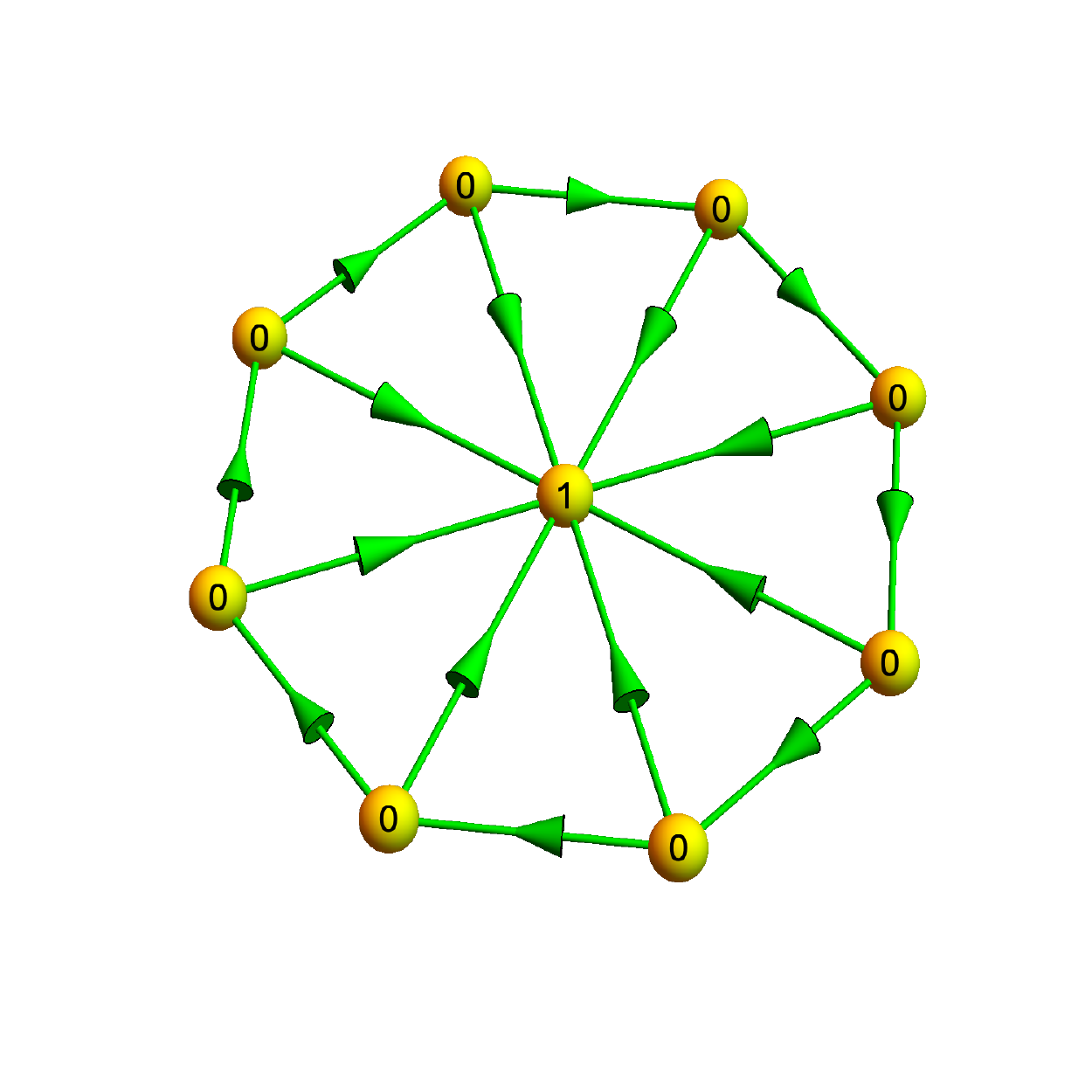}}
\scalebox{0.45}{\includegraphics{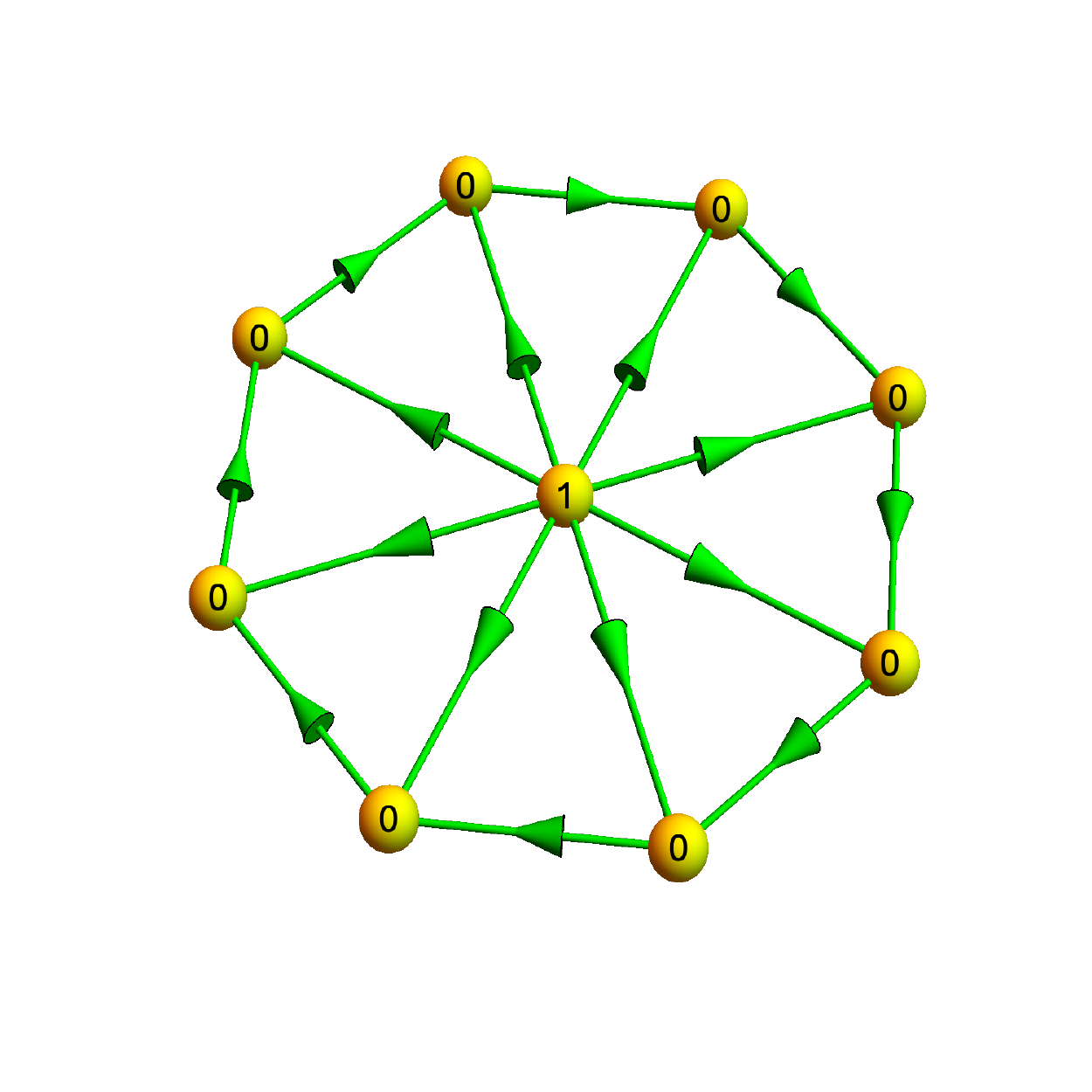}}
\label{triangle}
\caption{
For a sink and source in a 2-graph, the index is $1$. 
In the sink case, the incoming graph is the unit sphere
which has Euler characteristic $0$. In the source case, 
the incoming graph is the empty graph which has zero
Euler characteristic. 
}
\end{figure}

\begin{figure}[!htpb]
\scalebox{0.45}{\includegraphics{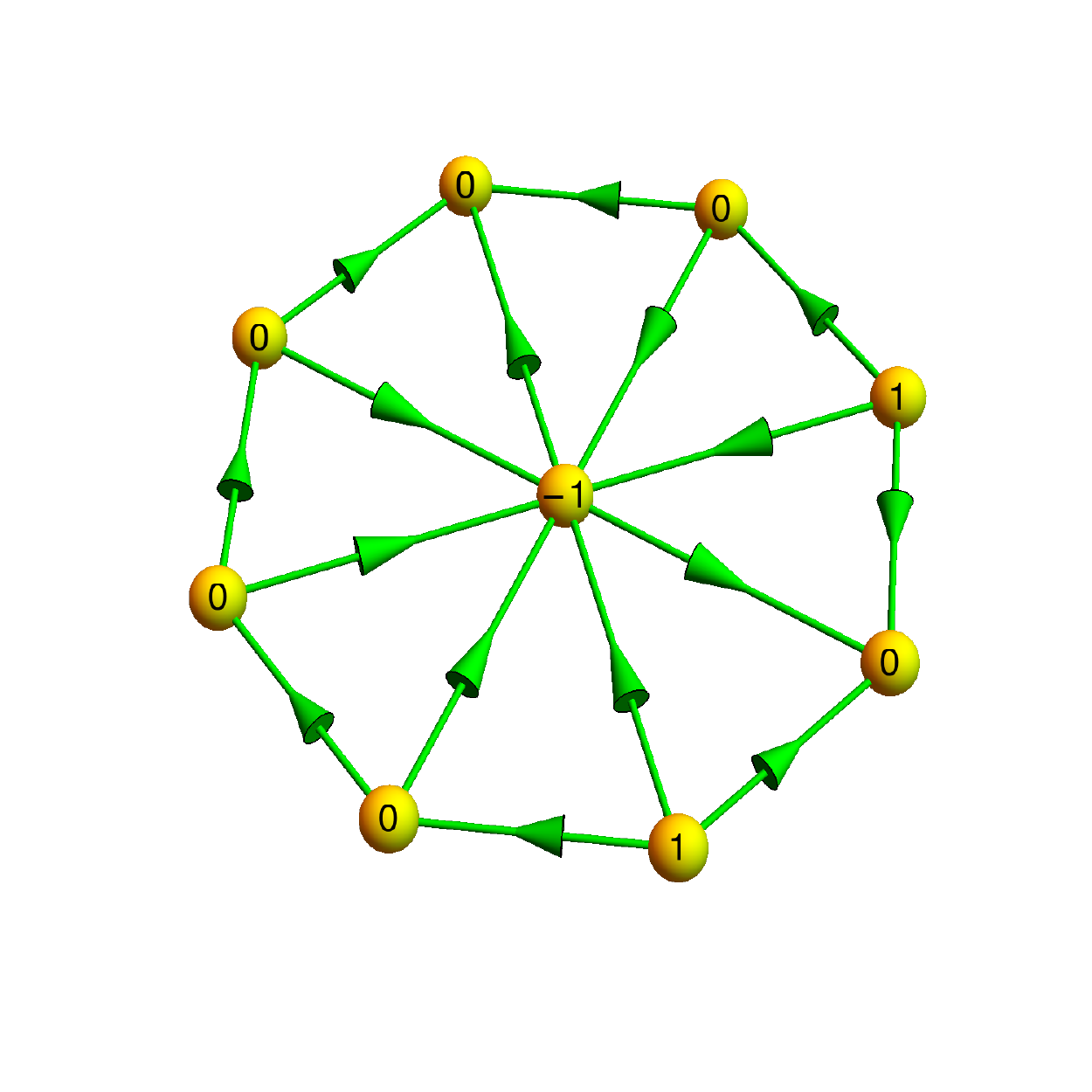}}
\scalebox{0.45}{\includegraphics{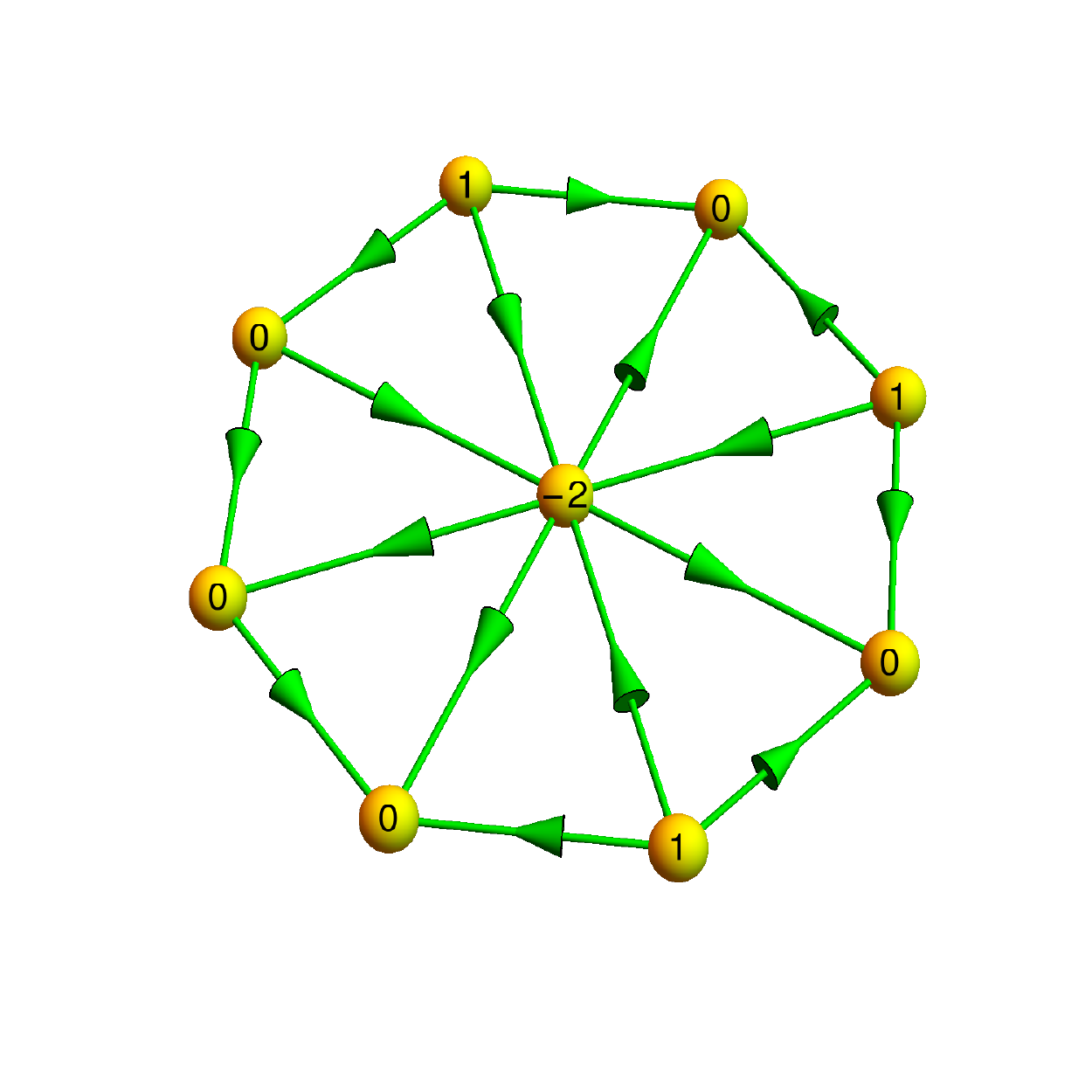}}
\label{triangle}
\caption{
For a saddle on a $2$-graph, the index is negative. 
We see a Morse saddle to the left and a Monkey saddle 
of index $-2$ to the right. There are then three incoming
and three outgoing connected components in $S(x)$.  
}
\end{figure}

\begin{figure}[!htpb]
\scalebox{1.2}{\includegraphics{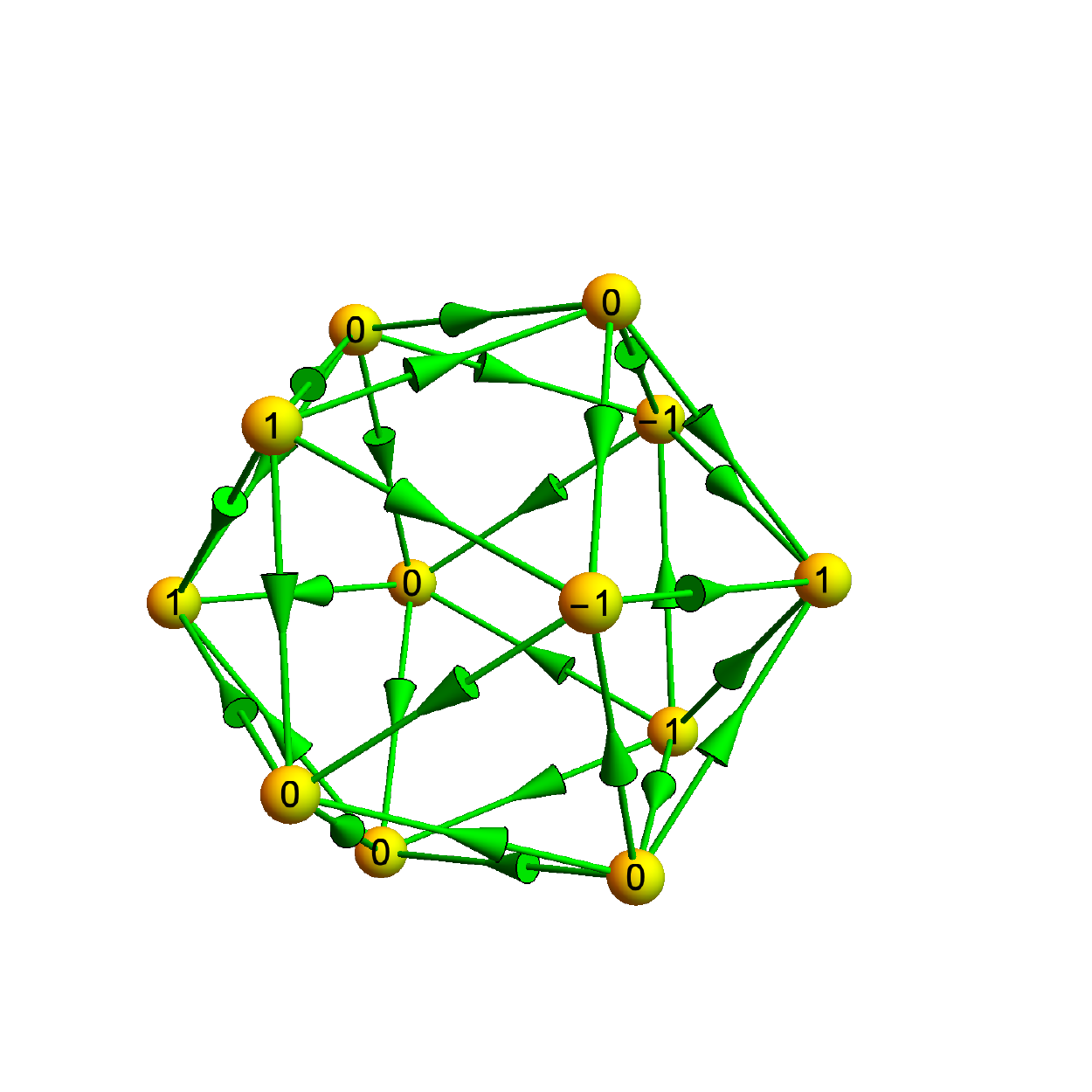}}
\label{triangle}
\caption{
A directed icosahedron graph with 4 equilibrium points of index 1
and two saddle equilibria of index $-1$. The Euler characteristic is
$2$, as it should be for 2-spheres.
}
\end{figure}

\section{Curvature}

\paragraph{}
In a broad way, curvature can be defined integral geometrically as index expectation. 
Gauss-Bonnet is then a direct consequence of Poincar\'e-Hopf, not requiring any proof. 
In the discrete, this has been explored in 
\cite{cherngaussbonnet, indexexpectation,indexformula,eveneuler}.
The definition has the advantage that can be used as a {\bf definition} both
in the discrete as well as in the continuum. 
While indices are integers and so divisors, curvature $K$ is in general 
real-valued. It still satisfies the Gauss-Bonnet constraint, assuring that the 
total curvature is Euler characteristic. \\

We can so study the question which spaces allow for constant curvature and explore
the first Hopf conjecture in the discrete: the later
is the question whether positive sectional curvature implies positive Euler characteristic
for even-dimensional discrete manifolds. Having the same question in the discrete allows a 
different avenue in exploring this notoriously difficult question for Riemannian manifolds. 
We will write about this oldest open problem in global differential geometry more elsewhere.
Of course, {\bf positive curvature manifolds} are then defined as a pair $(M,\mu)$ where
the probability  measure $\mu$ on Morse functions $\Omega$ has the property that
for any point $x \in M$ and any any two dimensional plane in $T_xM$, the manifold 
$N=\exp_x(D)$ has positive $\mu$-curvature at $x$. This makes sense as almost all 
Morse functions on $M$ induce Morse functions on $N$ near $x$. 

\paragraph{}
If $G=(V,E)$ is a finite simple graph. Assume that a probability measure $\mu$ on 
the space $\Omega$ of edge directions $F$ is given. This defines a probability measure 
$p_x$ on each simplex $x$. The probability $p_x(v)$ is the probability that $v$ is 
the largest element in $x$. The curvature $K(v) = E[i_g(v0]$ is the expectation 
of index functions $i_g(v)$. It satisfies 

\begin{thm}[Gauss-Bonnet]
$\sum_{ v \in V} K(v) = \chi(G)$
\end{thm}
\begin{proof}
Just take the expectation on both sides of the Poincar\'e-Hopf
identity $\chi(G) = \sum_{v \in V} i_g(v)$. 
\end{proof}

\paragraph{}
One interesting question we did not explore yet is how how big the probability of 
irrotational directed graphs are in the space of all directed Erd\"os-R\'enyi graphs 
with $n$ vertices and edge probability $p$ and how to construct irrotational directions
more generally than using potentials $g:V \to \mathbb{R}$ and defining $v \to w$
if $g(v) < g(w)$. For $p=1$, we look at complete graphs
of $n$ elements which has $m=n(n-1)/2$ edges and $2^m$ directions. We expect for
each triangle to have a probability $6/2^3=3/4$ to produce no cycle of length $3$. As
triangles appear with probability scaling $p^3$ and each triangle is with probability 
$1/4$ cyclic, we expect cyclic triangles to appear with a frequency with which 
triangles appear when the probability is $p \cdot (3/4)^{1/3} \sim 0.908 p$. 

\section{The Continuum}

\paragraph{}
Let $M$ be a smooth compact Riemannian manifold and let $\mu$ be 
a probability measure on smooth vector fields $F$ such that for $\mu$-almost all 
fields $F$, there are only finitely many hyperbolic equilibrium points. The later 
means that the Jacobean $dF(x)$ is invertible at every of the finitely many 
equilibrium point. We can then get a curvature on $M$ by taking the index 
expectation of these vector fields. 

\paragraph{}
For example, if $\mu$ is a measure on the set of Morse functions $g$ of a Riemannian 
manifold $M$, then the Hessian at a critical point of a function $g$ is the Jacobean of 
$F={\rm grad}(g)$. 
The Morse condition assures that the Hessians are invertible. The index expectation $K$
satisfies automatically Gauss-Bonnet. Such measures always exist.
Can every smooth function $K$ be realized as index expectation
if the Euler characteristic constraint is satisfied? 
More precisely, given a smooth function $K$ on $M$ 
satisfying $\int_M K \; dV = \chi(G)$, is there a measure on Morse functions such that $K$ 
is the index expectation? What we can show so far is that on any Riemannian manifold 
$M$ there is a measure $\mu$ on Morse functions which produces a constant curvature $K$. 
This will be explored a bit elsewhere. 

\paragraph{}
How do we get the Gauss-Bonnet-Chern integrand, the curvature $K$ which produces the 
Euler characteristic $\int_M K(x) \; dV(x) = \chi(M)$ as a Pfaffian of a Riemannian curvature
tensor expression? One possibility is to 
Nash \cite{EssentialNash} embed the Riemannian manifold into a finite dimensional Euclidean space $E$ and to 
take the probability space $\Omega$ of linear functions on $E$ which is a finite dimensional 
manifold and carries a natural unique rotational invariant measure $\mu$. 
Almost every function $g$ is known to be Morse with respect to this measure. It defines 
so Poincar\'e-Hopf indices $i_g(x)$ for almost every $g$. 
This leads to a curvature $K(x) = E[i_g(x)]$ which satisfies the 
{\bf theoremum egregium} (meaning that it is independent of the embedding). It also agrees with
 the Euler curvature: 

\begin{thm}
The expectation $K(x) = E_{\mu}[i(x)]$ is the Gauss-Bonnet-Chern 
integrand. It is by construction independent of the embedding. 
\end{thm}
\begin{proof}
This has been proven using Regge calculus \cite{Regge,Misner,CheegerMuellerSchrader}.
The Cheeger-Mueller-Schrader paper which also gives a new proof of Gauss-Bonnet-Chern.
For a proof of Patodi, see \cite{Cycon}.
The Regge approach means making a piecewise linear 
approximations of the manifold, defining curvature integral geometrically and 
then show that the limit converges as measures. \cite{CheegerMuellerSchrader}.
\end{proof}

\paragraph{}
The integral geometric point of view has been put forward earlier
in \cite{Banchoff1967}. The Regge approach allows to fit the 
discrete and continuum using integral geometry. The frame works of
Poincar\'e-Hopf curvature works equally well for polytopes and manifolds
and graphs. 

\paragraph{}
An other approach is to define {\bf axiomatically} what a ``good curvature"
on a Riemannian manifold should be. Let us assume that $M$ is a smooth,
compact Riemannian manifold and say that
$K$ is a ``good curvature" if the following properties hold:

\begin{itemize}
\item $K$ is a smooth function on $M$. 
\item Generalizes Gauss: In the two dimensional case, $K$ is the Gauss curvature.
\item In odd dimensions, $K$ is identically zero. 
\item Gauss-Bonnet: $\int_M K(x) \; dV(x) = \chi(M)$ is the Euler characteristic.
\item Theorema Egregium: $K$ does not depend on any embedding in an ambient space.
\item $K$  is local in the sense that $K(x)$ for $M$ is the same than the curvature
      of $K(x)$ when restricted to a small neighborhood of $x$. 
\end{itemize}

Added December 21: One could add $K_{M \times N}(x,y) = K_M(x) K_N(x)$, if $M,N$
are two even dimensional manifolds. 

\paragraph{}
We believe that in even dimensions, the Gauss-Bonnet-Chern integrand is the only choice for
these postulates. Since also the index expectation of a probability space of linear functions
in an ambient Euclidean space satisfies the postulates, this would establish the relation 
proven so far only via Regge calculus. Not everything with name curvature qualifies. Ricci
curvature, sectional curvature, mean curvature or scalar curvatures are of different type. 
Already curvature $|T'|/|r'|$ with $T=|r'|$ of a parametrization of a curve is an example
of a curvature which does not qualify, even without the odd dimension assumption. It does
not satisfy the Gauss-Bonnet formula for example.

\paragraph{}
A direct link between the continuum and discrete is given by internal set theory \cite{Nelson77}.
In that theory, a new attribute ``standard" is added to the standard axioms ZFC of set theory using 
three axioms. The advantage of internal set theory to other non-standard approaches is that
usual mathematics is untouched and that it is known to be a consistent extension of ZFC. 
A language extension allows for powerful shortcuts in real
analysis and or probability theory. For the later see the astounding monograph \cite{Nelson}. 
Borrowing terminology of that book, one could name discrete 
graph theoretical approaches to Riemannian geometry a ``radically elementary
differential geometry". It is extremely simple and works as follows:

\paragraph{}
In general for any mathematical object, there exists a {\bf finite set} which contains all the 
standard elements in that object. 
(See Theorem 1.2 in \cite{Nelson77}). In particular, 
given a compact manifold $M$, there exists a finite set $V$ which 
contains all standard points of $M$. 
This set is the {\bf vertex set} of a graph. 
Let $\epsilon>0$ be any positive number we can say that two points 
in $M$ are called ``connected" if their distance is smaller than $\epsilon$.
Now define the edge set $E$ as the set of pairs $(x,y)$ such that 
$x$ and $y$ are connected. This defines a finite simple graph. 
It depends on $\epsilon$. 
(One can not define it without specifying some $\epsilon$ by using
pairs which are infinitesimally close. The axiom (S) in the IST 
requires the relation $\phi$ to be internal. 
This is a classical mistake done when using non-standard
analysis and it was done in the first version of this paper. 
Thanks to Michael Katz to point this error out to me). 
Here are the three axioms IST of Nelson which extend ZFC:

\begin{eqnarray*}
(I)&:& (\forall^{st, fin} z \exists x \forall y \in z \ \phi(x,y)) \Leftrightarrow (\exists x \forall^{st} y \ \phi(x,y)), \; \phi \ internal \\
(S)&:& \forall^{st} x \exists^{st}b \forall^{st}x \ (x \in b \Leftrightarrow x \in a \ and \ \phi(x)), \; \phi \ arbitrary \\
(T)&:& (\forall^{st}x \  \phi(x,u)) \Leftrightarrow (\forall x \ \phi(x,u)), \; \phi \ internal, \; u \ standard \\
\end{eqnarray*}

\paragraph{}
There are now Poincar\'e-Hopf or Gauss-Bonnet formulas
available and they lead to the same results if $\epsilon$ is infinitesimal. 
The graph $(V,E)$ is naturally homotopic to any good triangulation of the manifold. It is of 
course of much larger dimension but we do not care. 
But the finite simple graph $(V,E)$ gives  more, it contains all information about the 
original manifold $M$. 

\paragraph{}
Can one recover the Riemannian metric? Not from the graph itself as homeomorphic 
manifolds can be modeled by graph isomorphic graphs. We need more structure.
There are two ways to compute distances in the graph: 
use the Connes formula \cite{Connes} from the Dirac operator $D=d+d^*$ defined by 
the exterior derivative $d$ defined on the graph has an advantage that it can be deformed \cite{DiracKnill}. 
An other is to define the distance between two points $x,y$ as the geodesic distance between two points in the graph and 
scale this so that the diameter of $G$ is the same than the diameter of the manifold. 
The measure $\mu$ on Morse functions which produces the Euler curvature however 
should allow to recover the distance as Crofton formulas allow to define a length
of curves integral geometrically and so define a notion of geodesic. 

\section{Illustration} 

\paragraph{}
The simplest curvature is the {\bf signed curvature} $\frac{d}{dt} {\rm arg}(r'(t))$ of 
a planar curve $r(t)$. Gauss-Bonnet in that 
case goes under the name {\bf Hopf Umlaufsatz} which tells that for a simple closed 
$C^2$ curve, the total signed curvature is $2\pi$. We mention the Hopf proof \cite{hopf35}
because the index expectation definition of curvature allows deformations. Since Hopf 
could handle a subtle case with elegance using deformation, there is an obvious question
whether one can deform the measure $\mu$ defining curvature to make it positive if 
all sectional curvatures are positive. The Euler curvature has long been known not to 
be positive if curvature is positive. 

\paragraph{}
Hopf proved this through a homotopy argument. 
Assume $r(t)$ is parametrized as $r:[0,1] \to \mathbb{R}^2$. Given a path
from $(0,0)$ to $(1,1)$ he looked at the total angle change which must be a
multiple of $2\pi$. As the total change depends continuously on the curve, homotopic curves
from $(0,0)$ to $(1,1)$ have the same total change. Going from $(0,0)$ to $(1,0)$ gives $\pi$.
Going from $(1,0)$ to $(1,1)$ again gives $\pi$. So that the piecewise linear path from $(0,0)$ 
to $(1,1)$ gives a change of $2\pi$.

\begin{figure}[!htpb]
\scalebox{0.6}{\includegraphics{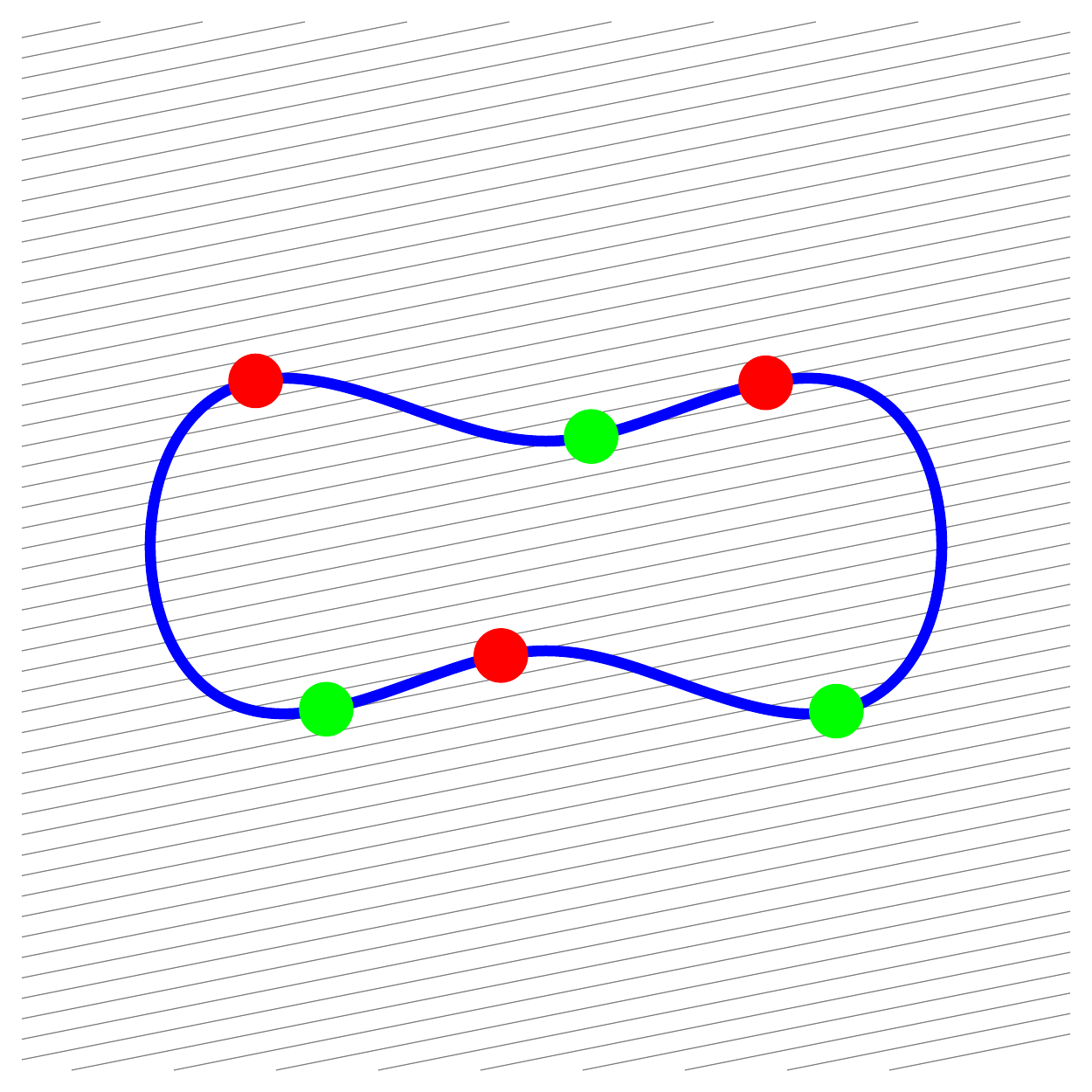}}
\label{hopf}
\caption{
Embedding a manifold $M$ into an ambient space and
computing the Morse indices of Morse functions obtained by 
a linear function in the ambient space produces indices.
Averaging over the projective space gives signed curvature.
Integrating over the entire circle gives 0 as every maximum
matches a minimum. 
}
\end{figure}

\paragraph{}
In the case of a one dimensional manifold, maxima have index $-1$
and minima have index $1$. The Puiseux formula shows that 
the expectation value of the index is the signed curvature 
$K(t)=r'(t) \times r''(t)/|r'(t)|^3$. Gauss-Bonnet is then the 
Hopf Umlaufsatz. 
The Gauss-Bonnet-Chern integrand is a priori not defined in the 
odd-dimensional case but there is an integral geometric invariant.

\begin{figure}[!htpb]
\scalebox{0.5}{\includegraphics{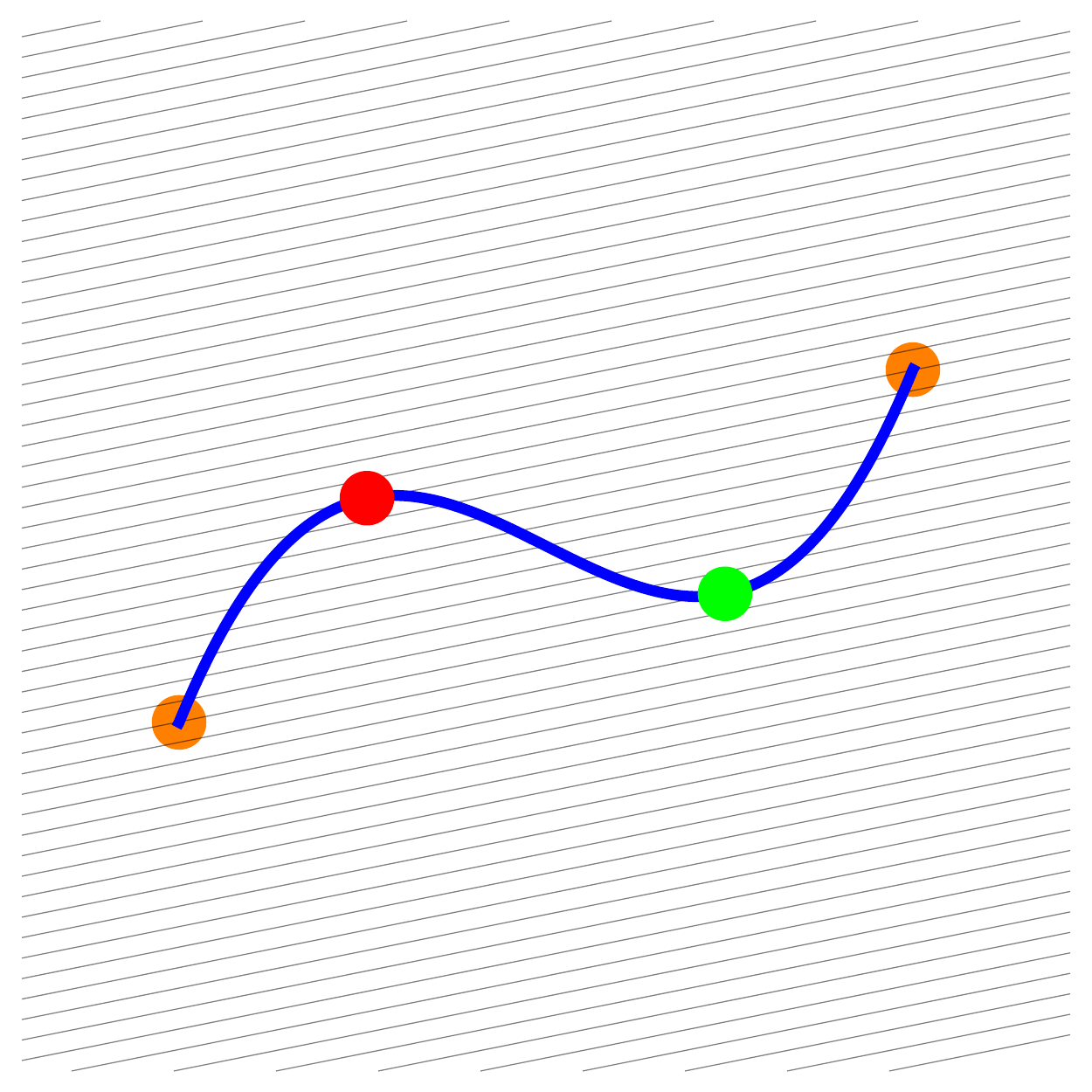}}
\label{hopf}
\caption{
For a manifold with boundary, the Poincar\'e-Hopf curvature gives
a boundary curvature. In the case of a one-dimensional 
manifold with boundary part of the curvature is on the boundary. 
The total curvature of a curve is $1$. If the measure is taking the
full $2\pi$ turn, then the curvature $1/2$ is supported on both ends
and zero inside. 
}
\end{figure}

\begin{figure}[!htpb]
\scalebox{0.5}{\includegraphics{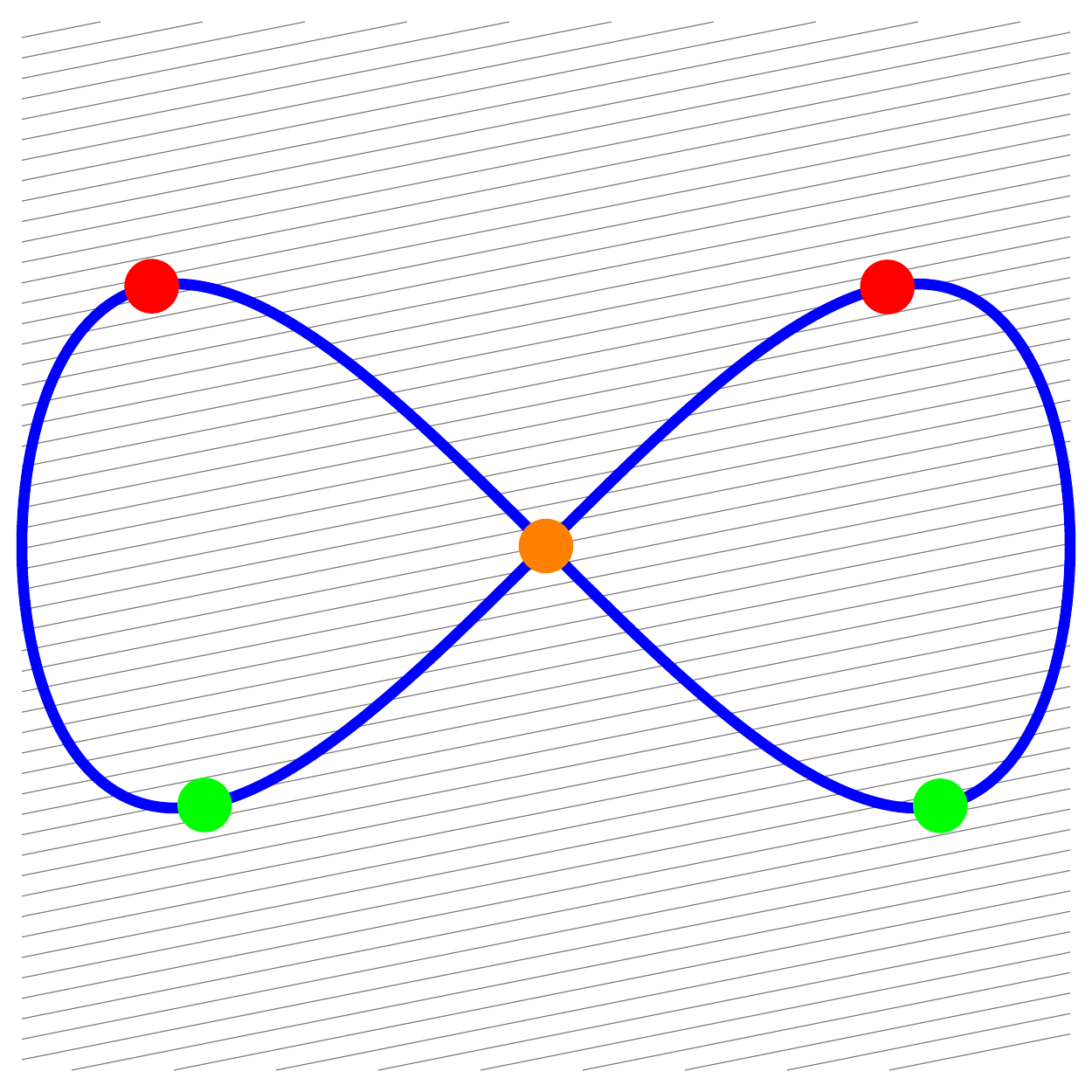}}
\label{hopf}
\caption{
For an odd dimensional variety $M$, the Poincar\'e-Hopf curvature 
is supported on the singularities. In this case of a figure 8 curve
in the form of the lemniscate $y^2-x^2+x^4=0$, the index expectation
is the Dirac point measure with weight $-1$ at $(0,0)$ (the indices at
other points cancel out like for a closed curve). The Euler
characteristic of $M$ is $-1$. In order that the Poincar\'e-Hopf index
to exist, we only need that for small enough $r$, the spheres $S_r(x)$
of a singular point are of the same class but smaller dimensional. It 
is enough for example to assume that the variety has the property that
$S_r(x)$ is a manifold for small enough $r$.
}
\end{figure}

\begin{figure}[!htpb]
\scalebox{0.6}{\includegraphics{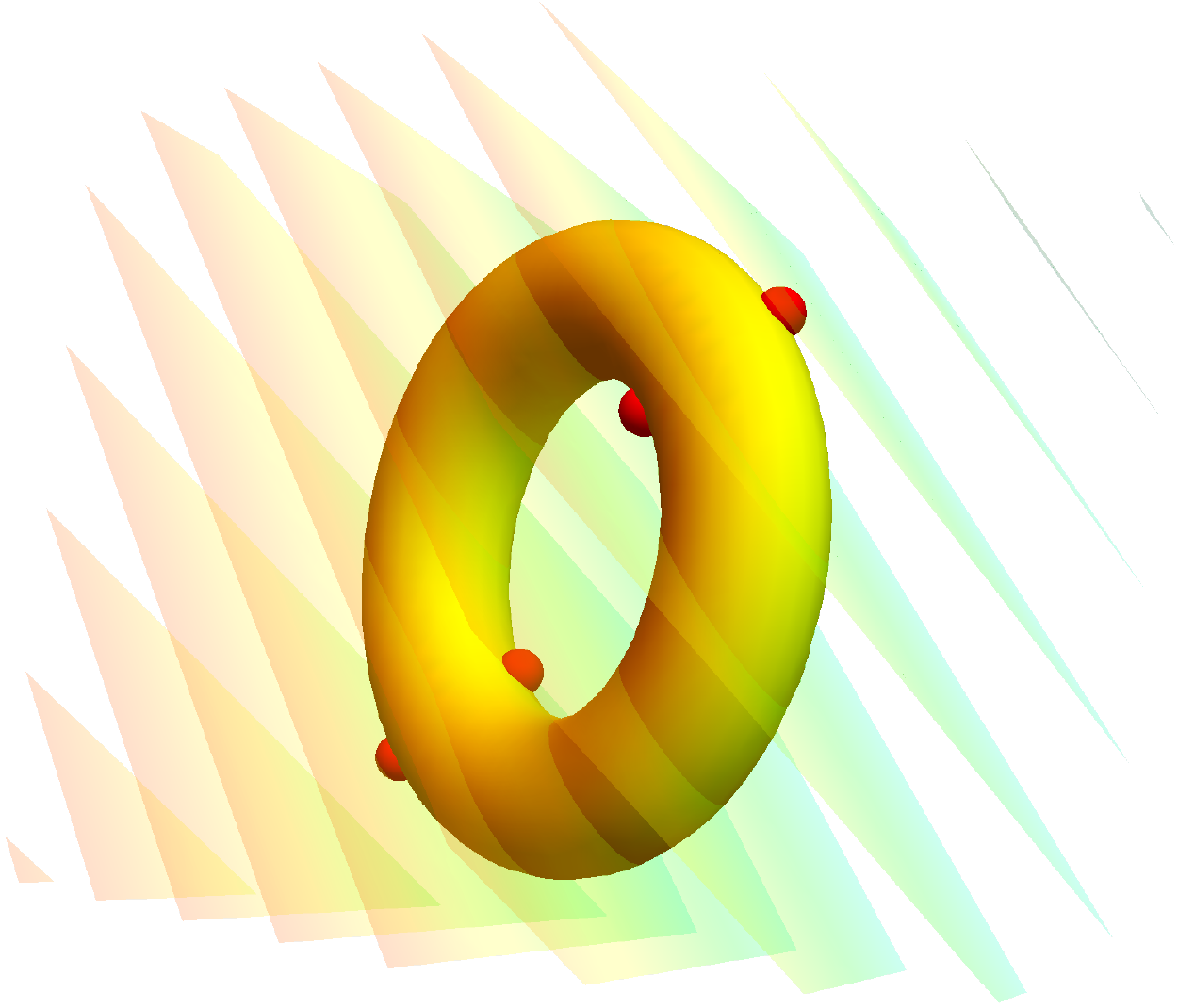}}
\label{hopf}
\caption{
Unlike for curves, in higher dimensions, one in general 
needs larger dimensional spaces for embedding. Here is an embedding
of a torus in $\mathbb{R}^3$ for which the index expectation gives the 
standard Gauss curvature. 
}
\end{figure}

\begin{figure}[!htpb]
\scalebox{0.6}{\includegraphics{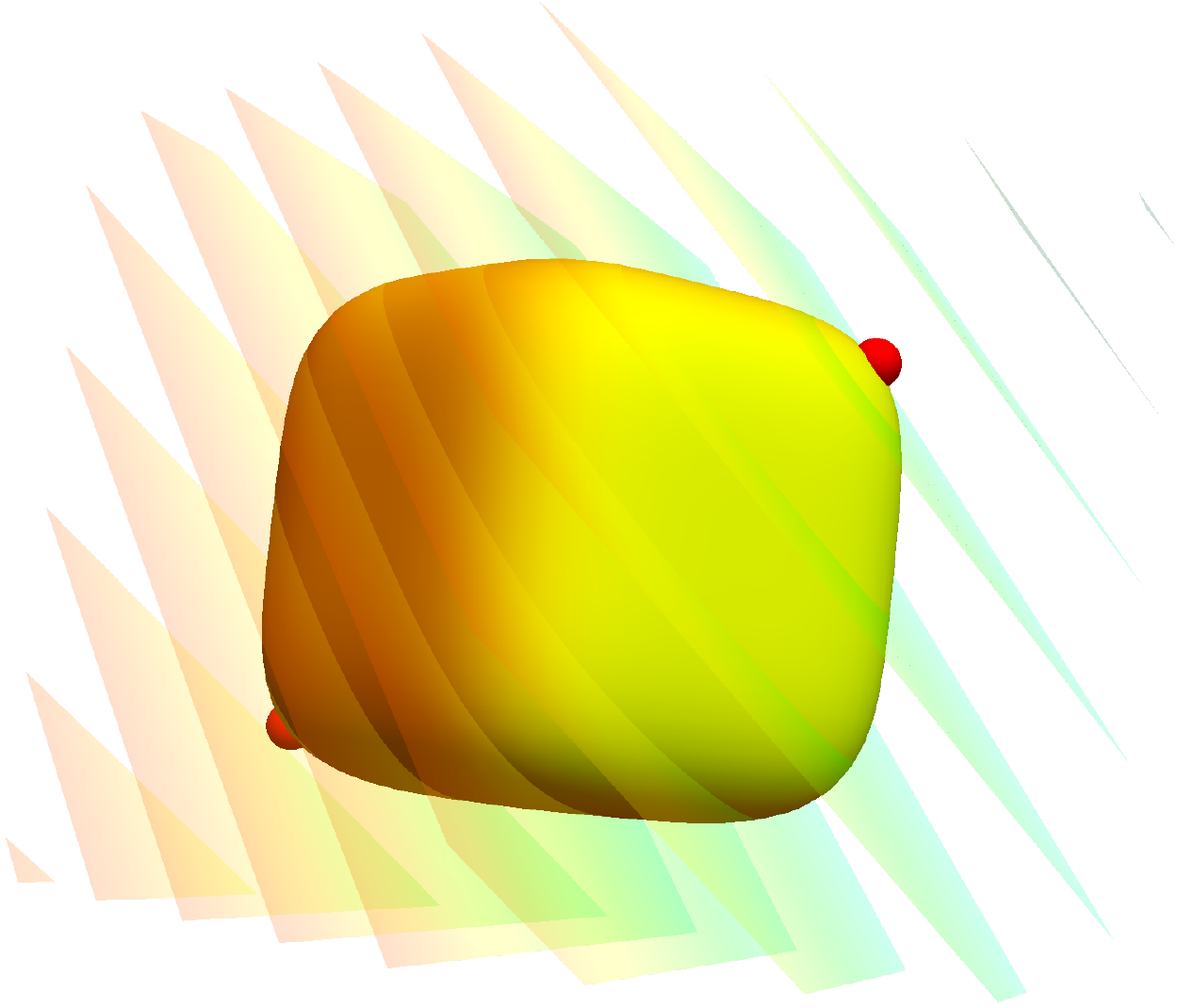}}
\label{hopf}
\caption{
For spheres, there are functions with exactly two Poincar\'e-Hopf critical points.
}
\end{figure}

\begin{figure}[!htpb]
\scalebox{0.6}{\includegraphics{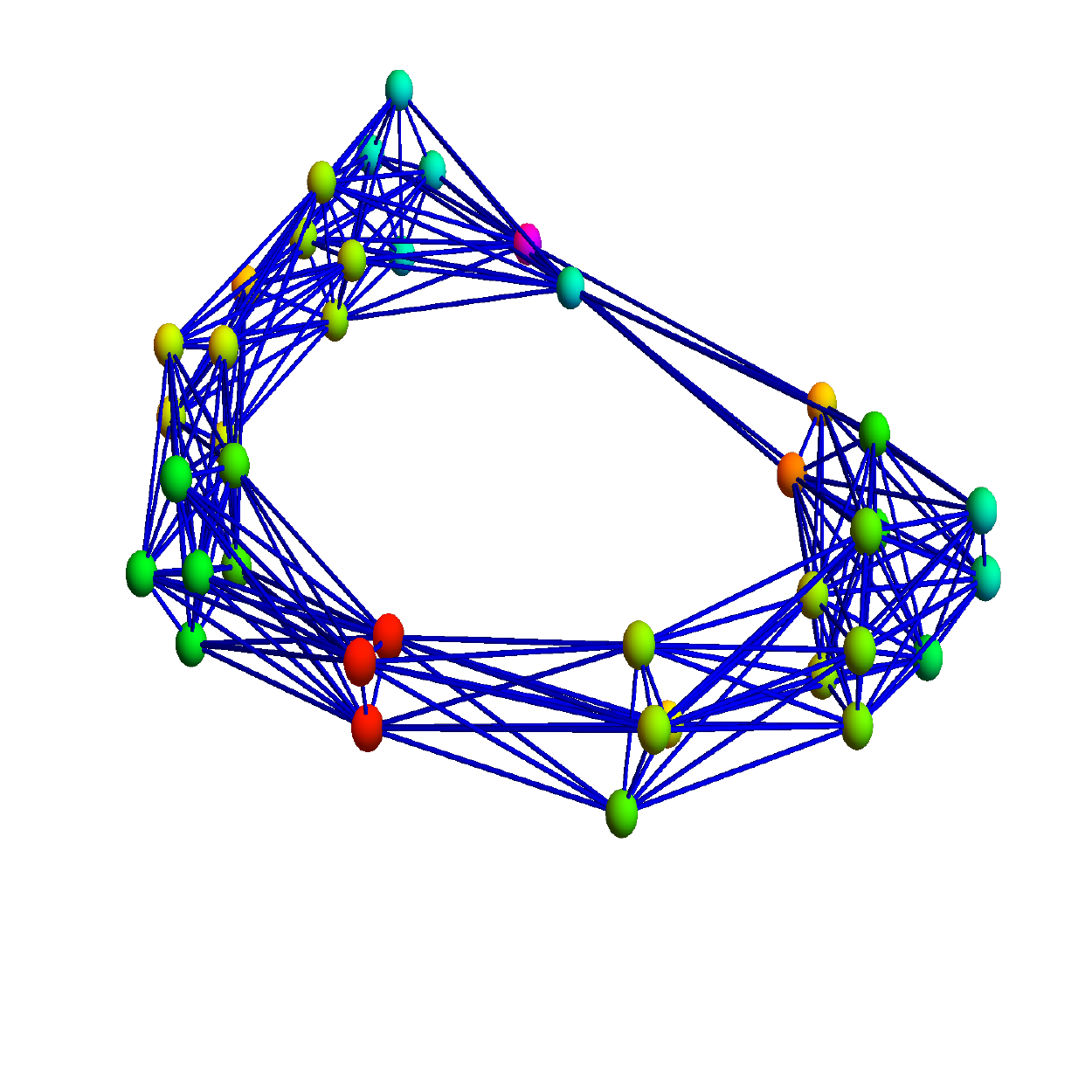}}
\scalebox{0.6}{\includegraphics{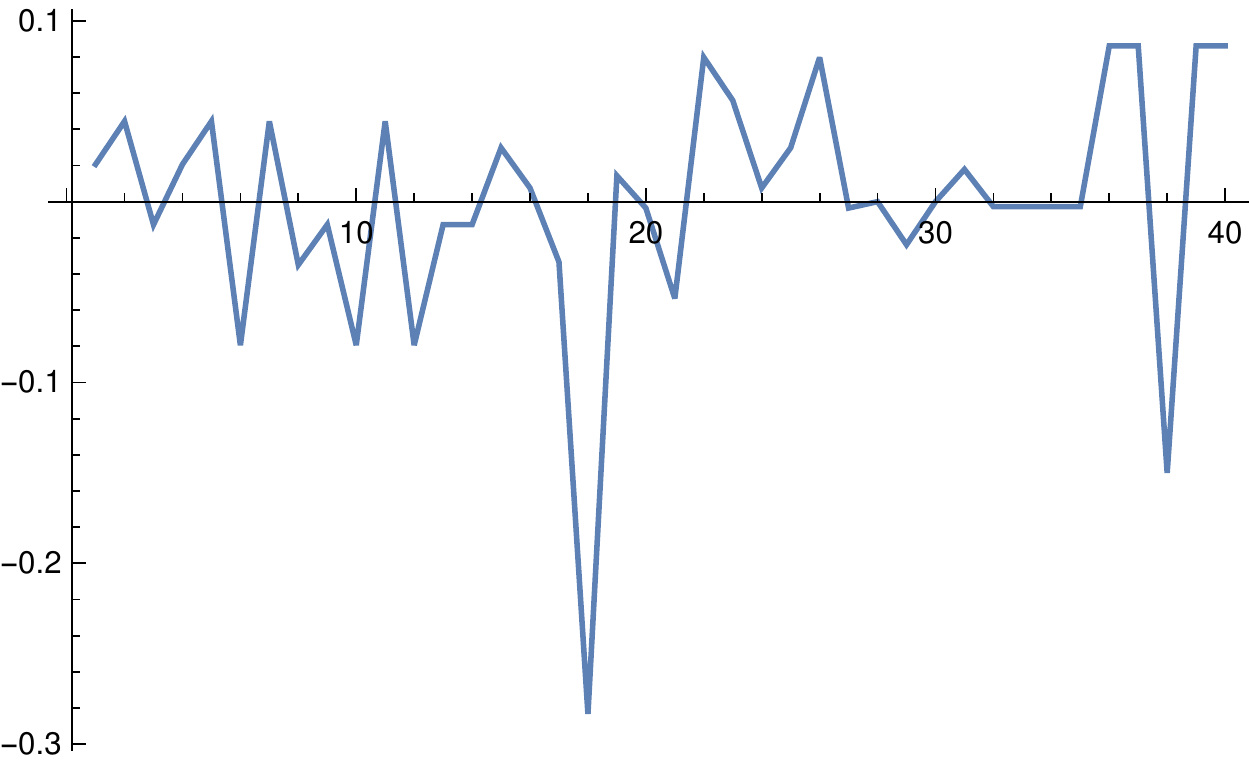}}
\label{nonstandardcircle}
\caption{
The curvature of a nonstandard manifold $M$ is the Euler curvature. We see
a picture illustrating a non-standard circle. The figure to the right
shows the curvatures. They are very small. If the vertex set is a finite
set containing all standard elements of $M$, then the standard part of 
the curvature of the graph is the Euler curvature of the manifold. It is zero
for odd dimensional manifolds. 
}
\end{figure}

\bibliographystyle{plain}

\end{document}